\documentclass{irmamodif}

\usepackage{mathrsfs,eucal,enumerate,url}
\usepackage{epsfig,psfrag}

\theoremstyle{plain}

\newtheorem{lemma}{Lemma}[section]

\theoremstyle{definition}

\newcommand{\card}{\operatorname{card}}
\newcommand{\End}{\operatorname{End}}
\newcommand{\Bili}{\operatorname{Bil}_\bA}

\newcounter{parag}[section]

\newcommand{\dst}{\displaystyle}

\newcommand{\ens}{\enspace}

\newcommand{\al}{\alpha}
\newcommand{\alb}{\text{\boldmath{$\alpha$}}}
\newcommand{\be}{\beta}
\newcommand{\beb}{\text{\boldmath{$\beta$}}}
\newcommand{\ga}{\gamma}
\newcommand{\gab}{\text{\boldmath{$\gamma$}}}

\newcommand{\de}{\delta}
\newcommand{\deb}{\text{\boldmath{$\delta$}}}

\newcommand{\De}{\Delta}

\newcommand{\la}{\lambda}

\newcommand{\om}{\omega}
\newcommand{\omb}{\text{\boldmath{$\om$}}}

\newcommand{\sig}{{\sigma}}

\renewcommand{\th}{{\theta}}
\newcommand{\Th}{\Theta}
\newcommand{\ph}{\varphi}

\newcommand{\ze}{{\zeta}}

\newcommand{\ao}{\{\,}          
\newcommand{\af}{\,\}}          

\newcommand{\ID}{\operatorname{Id}}

\newcommand{\conc}{\raisebox{.15ex}{$\centerdot$}}

\newcommand{\bul}{\bullet}

\newcommand{\pa}{\partial}
\newcommand{\na}{\nabla}
\newcommand{\ii}{^{-1}}

\newcommand{\est}{\emptyset}

\newcommand{\ceil}[1]{\lceil #1 \rceil}

\newcommand{\sh}[3]{\operatorname{sh}\begin{pmatrix}#1,\, #2\\#3\end{pmatrix}}
\newcommand{\tsh}[3]{\operatorname{sh}{\big( \begin{smallmatrix}#1,\,
#2\\#3\end{smallmatrix} \big)}}

\newcommand{\val}{\operatorname{val}}
\newcommand{\vl}[1]{\val\left(#1\right)}

\newcommand{\ie}{{i.e.}}

\newcommand{\eg}{{e.g.}}

\newcommand{\wrt}{{with respect to}}

\newcommand{\rhs}{{right-hand side}}

\newcommand{\bA}{\mathbf{A}}

\newcommand{\bB}{\mathbf{B}}
\newcommand{\C}{\mathbb{C}}

\newcommand{\N}{\mathbb{N}}

\newcommand{\R}{\mathbb{R}}

\newcommand{\Z}{\mathbb{Z}}

\newcommand{\cN}{\mathcal{N}}
\newcommand{\cR}{\mathcal{R}}

\newcommand{\cU}{\mathcal{U}}
\newcommand{\cV}{\mathcal{V}}

\newcommand{\cVt}{{\cV\hspace{-.45em}\raisebox{.35ex}{--}}}

\newcommand{\dd}{{\mathrm d}}

\newcommand{\ee}{{\mathrm e}}

\newcommand{\wh}{\widehat}
\newcommand{\wt}{\widetilde}

\newcommand{\htb}[1]{\raisebox{-.08ex}{${\stackrel{
            \raisebox{-.18ex}{$\scriptscriptstyle\wedge$}
          }{#1}
     }$}}

\newcommand{\norm}[1]{\Vert#1\Vert}

\newcommand{\gA}{\mathscr A}

\newcommand{\gF}{\mathscr F}

\newcommand{\gM}{\mathfrak M}

\newcounter{parage}
\renewcommand{\theparage}{\thesection.{\arabic{parage}.}}
\newcommand{\parage}{\medskip \addtocounter{parage}{1} 
\noindent{\theparage\ } }

\begin{document}


\address{
Institut de m\'ecanique c\'eleste, CNRS\\ 
77 av.\ Denfert-Rochereau, 75014 Paris, France\\
email:\,\tt{sauzin@imcce.fr}
}


\title{Initiation to mould calculus through the example of saddle-node singularities}

\author{David Sauzin}


\maketitle

\begin{abstract}
This article proposes an initiation to \'Ecalle's mould calculus, a
powerful combinatorial tool which yields surprisingly explicit formulas for the
normalising series attached to an analytic germ of singular vector field.
This is illustrated on the case of saddle-node singularities,
generated by two-dimensional vector fields which are formally conjugate to
Euler's vector field $x^2\frac{\pa}{\pa x}+(x+y)\frac{\pa}{\pa y}$,
and for which the formal normalisation proves to be resurgent in~$1/x$.
\end{abstract}


\setcounter{section}{-1}

\section{Introduction}


This article is a survey of a part of a longer article \cite{irma} which aims at
presenting in a systematic way some aspects of \'Ecalle's theory of moulds, 
with the example of the classification of saddle-node singularities as a red thread.

Mould calculus was developed by J.~\'Ecalle in relation with his Resurgence
theory almost thirty years ago \cite{Eca81}
and the application to the saddle-node problem was indicated in \cite{Eca84} in
concise manner.

Here, we omit much of the material of \cite{irma} and present the arguments in a
different order, trying to explain the formal part of \'Ecalle's method in fewer
pages and hoping to arouse the reader's interest by this illustration of
mould calculus.




\section{Saddle-node singularities}


Germs of holomorphic singular foliations of $(\C^2,0)$ are defined by analytic
differential equations of the form
$P(x,y)\dd y - Q(x,y)\dd x = 0$ with $P$ and $Q\in\C\{x,y\}$ both vanishing at
the origin.
Classifying such singular foliations means describing the conjugacy classes under
the action of the group of germs of analytic invertible transformations of
$(\C^2,0)$; this is equivalent to classifying the corresponding singular vector
fields
$P(x,y)\frac{\pa\,}{\pa x} + Q(x,y)\frac{\pa\,}{\pa y}$ up to time-change.

\emph{We shall examine the case in which the foliation is assumed to be {formally}
conjugate to the standard saddle-node foliation, defined by
$x^2\dd y - y\dd x=0$.}
This is the simplest case for differential equations whose $1$-jets admit
$0$ and~$1$ as eigenvalues (the possible formal normal forms are $x^{p+1}\dd
y - (1+\la x)\dd x=0$, with $p\in\N^*$ and $\la\in\C$, and $y\,\dd x=0$).

In this case, by a classical theorem of Dulac (\cite{MR}, \cite{Moussu}), the
foliation can be analytically reduced to the form
$x^2\dd y - A(x,y)\dd x=0$, with
\begin{equation}	\label{eqassA}
A(x,y) \in\C\{x,y\}, \quad
A(0,y) = y, \quad
\frac{\pa^2 A}{\pa x\pa y}(0,0) = 0.
\end{equation}
Moreover, the vector fields corresponding to the foliation and to the normal
form,

\begin{equation}	\label{defXX}
X = x^2 \frac{\pa\,}{\pa x} + A(x,y) \frac{\pa\,}{\pa y}
\quad\text{and}\quad
X_0 = x^2 \frac{\pa\,}{\pa x} + y \frac{\pa\,}{\pa y},
\end{equation}
are themselves formally conjugate:
there is a unique formal transformation of the form
\begin{equation}	\label{eqdefthph}
\th(x,y) = \big(x, \ph(x,y) \big), \qquad 
\ph(x,y) = y + \sum_{n\ge0} \ph_n(x) y^n, \qquad 
\ph_n(x) \in x\C[[x]].
\end{equation}
such that 
\begin{equation}	\label{eqconjugeq}
X = \th^* X_0
\end{equation}
(no time-change is needed here).
We call it the \emph{formal normalisation} of the foliation (or of~$X$ itself).

Thus, in this case, our problem of analytic classification boils down to
describing the analytic conjugacy classes of vector fields of the form~$X$ under
the action of the group of ``fibred'' transformations (leaving~$x$ unchanged),
knowing that they are all formally conjugate of one another (since all of
them are formally conjugate to~$X_0$).

Our starting point will be the data~\eqref{eqassA}--\eqref{defXX}.
We shall study the formal conjugacies~\eqref{eqdefthph} by
means of \'Ecalle's ``moulds'' and see how this leads to the ``resurgent''
character of the formal series which appear.
(A similar study could be performed for the more general normal forms, with
any~$p$ and~$\la$.)
We shall not give here the complete resurgent solution of the problem of
analytic classification; the reader is referred to \cite{irma} for this and for
the comparison with Martinet-Ramis's solution \cite{MR}.


\section{Formal separatrix, formal integral}


Observe that our foliations always have an analytic separatrix (a leaf passing
through the origin), namely the curve $\{x=0\}$ in Dulac coordinates.
For the vector field~$X$, this corresponds to the solution
$z\mapsto(0,u\ee^z)$
(with an arbitrary constant of integration $u\in\C$).

The ``formal curve'' $\{y=\ph_0(x)\}$ is to be considered as a ``formal
separatrix'' of our foliation (and as a formal centre manifold of~$X$), since
it is the image by~$\th$ of $\{y=0\}$ which is the other separatrix of the
normal form.
A time-parametrisation of the corresponding integral curve of~$X_0$ is
$z \mapsto (x,y)=(-1/z,0)$.
The formal series~$\ph_0$ can thus be obtained by setting $\wt\ph_0(z) =
\ph_0(-1/z)$ and looking for a formal solution $z \mapsto
\big(-1/z,\wt\ph_0(z)\big)$ of~$X$, 
\ie\ the formal series~$\wt\ph_0$ must solve the non-linear differential equation
\begin{equation}	\label{eqdiffeqX}
\frac{\dd\wt Y}{\dd z} = A(-1/z,\wt Y).
\end{equation}
More generally, if we set
\begin{equation}	\label{eqdefYzu}
\wt Y(z,u) = u \,  \ee^z + \sum_{n\ge0} u^n \ee^{nz} \wt\ph_n(z), \qquad 
\wt\ph_n(z) = \ph_n(-1/z) \in z\ii\C[[z\ii]],
\end{equation}
we get what \'Ecalle calls a ``formal integral'' of~$X$, 
a formal object containing a free parameter and solving~\eqref{eqdiffeqX}
(this is equivalent to finding a formal transformation of the
form~\eqref{eqdefthph} which solves~\eqref{eqconjugeq}).
One can find the formal series~$\wt\ph_n$ by solving the ordinary differential
equations obtained by expanding~\eqref{eqdiffeqX} in powers of~$u$.

The simple and famous example of Euler's equation, for which $A(x,y)=x+y$, shows
that the above formal series can be divergent.
Indeed, the equation for~$\ph_0(x)$ is then
$x^2\frac{\dd\ph_0}{\dd x} = x + \ph_0$
and one finds $\ph_0(x) = - \sum_{n\ge1} (n-1)! x^n$.
Since the equation is affine, there are no other non-trivial series in this
case:
$\ph(x,y)$ boils down to $y+\ph_0(x)$.

It is not by studying the differential equation~\eqref{eqdiffeqX} that we shall
get information on the formal transformation~$\th$,
but rather by directly working on the conjugacy equation~\eqref{eqdefthph}.


\section{The formal normalisation as an operator}


Let $\gA=\C[[x,y]]$ and let~$\nu$ denote the standard valuation
(thus $\nu(f)\in\N$ is the ``order'' of~$f$, with the convention that
$\nu(x^m y^n)=m+n$).
We denote by~$\gM$ the maximal ideal of~$\gA$, consisting of all formal series
without constant term, \ie\ $\gM=\ao f\in\gA\mid \nu(f)\in\N^*\af$.

We are given analytic vector fields~$X$ and~$X_0$ as in~\eqref{defXX}, which are
operators of the $\C$-algebra~$\gA$, more precisely $\C$-derivations
(in fact, they are derivations of~$\C\{x,y\}$, but we begin by forgetting analyticity).
We shall look for the formal normalisation~$\th$ through its ``substitution
operator'', which is the operator $\Th\in\End_\C\gA$ defined by
\begin{equation}	\label{eqdefThth}
\Th f = f\circ\th, \qquad f\in\gA.
\end{equation}

There is in fact a one-to-one correspondence beween ``substituable'' pairs of formal
series, \ie\ pairs $\th=(\th_1,\th_2)\in\gM\times\gM$, and operators of~$\gA$
which are \emph{formally continuous algebra homomorphisms}, \ie\ operators
$\Th\in\End_\C\gA$ such that $\Th(fg)=(\Th f)(\Th g)$ for any $f,g\in\gA$ and 
$\nu(\Th f_n)\to\infty$ for any sequence $(f_n)$ of~$\gA$ such that
$\nu(f_n)\to\infty$
(the proof is easy, see \eg\ \cite{irma}; one goes from~$\Th$ to~$\th$ simply by
setting $\th_1=\Th x$ and $\th_2=\Th y$).

Equation~\eqref{eqconjugeq} can be written
$$
(X f)\circ\th = X_0(f\circ\th), \qquad f\in\gA,
$$
and thus rephrased in terms of the substitution operator~$\Th$ as
\begin{equation}	\label{eqconjugOp}
\Th X = X_0 \Th.
\end{equation}
\emph{Thus, looking for a formal invertible transformation solution
of the conjugacy equation~\eqref{eqconjugeq} is equivalent to looking for a
formally continuous algebra automorphism solution of~\eqref{eqconjugOp}}.


\section{The formal normalisation as a mould expansion}	\label{secnormalmdlexp}


The general strategy for finding a normalising operator by mould calculus
consists in constructing it from the ``building blocks'' of the object one
wishes to normalise.
This implies that we shall restrict our attention to those operators of~$\gA$
which are obtained by combining the homogeneous components of~$X$ in all
possible ways\ldots

Since we are interested in fibred transformations, it is relevant to
consider the homogeneous components of~$X$ relatively to the variable~$y$ only
and to view $\bA = \C[[x]]$ as a ring of scalars;
we thus write
\begin{equation}	\label{eqdefBn}
X = X_0 + \sum_{n\in\cN} a_n B_n, \qquad
B_n = y^{n+1} \frac{\pa\,}{\pa y}, 
\end{equation}
with coefficients $a_n\in\bA$ stemming from the Taylor expansion
\begin{equation}	\label{eqdefiancN}
A(x,y) = y + \sum_{n\in\cN} a_n(x) y^{n+1}, \qquad
\cN = \ao n\in\Z \mid n \ge -1 \af.
\end{equation}
Homogeneity here means that each $a_n B_n$ sends $\bA y^k$ in $\bA
y^{k+n}$ for all $k\in\N$: the component $a_n B_n$ is homogeneous of degree~$n$
(this was the reason for shifting the index~$n$ by one unit in the Taylor
expansion), while $X_0$ is homogeneous of degree~$0$.

\emph{We shall look for a solution of~\eqref{eqconjugOp} among all the
operators of the form
\begin{equation}	\label{eqfirstmouldexp}
\Th = \sum_{r\ge0} \sum_{n_1,\dotsc,n_r\in\cN}
\cV^{n_1,\dotsc,n_r} B_{n_r}\dotsm B_{n_1}.
\end{equation}
Here $(\cV^{n_1,\dotsc,n_r})$ is a collection of coefficients in~$\bA$, to be
chosen in such a way that
formula~\eqref{eqfirstmouldexp} has a meaning as a formally continuous operator of~$\gA$ and defines an
automorphism solving~\eqref{eqconjugOp}.}
The above summation is better understood as a summation over~$\cN^\bul$, the
free monoid consisting of all words~$\omb$ (of any length~$r$) the letters of
which are taken in the alphabet~$\cN$.
We thus set
\begin{equation}	\label{eqdefbBn}
\bB_\omb = B_{n_r}  \dotsm B_{n_1}, \qquad
\omb = (n_1,\dotsc,n_r) \in \cN^\bul,
\end{equation}
and $\bB_\est=\ID$ for the only word of zero length ($\omb=\est$),
and rewrite formula~\eqref{eqfirstmouldexp} as
\begin{equation}	\label{eqmouldexponoid}
\Th = \sum_{\omb\in\cN^\bul} \cV^\omb \bB_\omb.
\end{equation}
An operator~$\Th$ defined by such a formula is called a mould expansion, or a mould-comould
contraction.
Here the \emph{comould} is the map 
$$
\omb\in\cN^\bul \mapsto \bB_\omb \in \End_\C\gA,
$$
that we decided to define from the homogeneous components of~$X$,
and the \emph{mould} is the map
$$
\omb\in\cN^\bul \mapsto \cV^\omb \in \bA,
$$
that we must find so as to satisfy the aforementioned requirements.




\section{The general framework for mould-comould contractions}	\label{secgnal}


The general theory of moulds and comoulds requires:
\begin{enumerate}[--]
\item an alphabet~$\cN$, which is simply a non-empty set, often with a structure of
commutative semigroup (since it appears in practice as a set of possible
degrees of homogeneity); 
\item a commutative $\C$-algebra~$\bA$ (the unit of which we denote by~$1$), in
which moulds take their values;
\item an $\bA$-algebra~$\gF$ (the unit of which we denote by~$\ID$), possibly
non-commutative, in which comoulds take their values.
We also assume that a \emph{complete ring pseudovaluation}
$\val\colon\gF\to\Z\cup\{\infty\}$ is given.\footnote{
This means that $\vl{\Th}=\infty \!\Longleftrightarrow\! \Th=0, \ens
\vl{\Th_1-\Th_2} \ge \min \big\{\vl{\Th_1},\vl{\Th_2}\big\}$, 
$\vl{\Th_1\Th_2} \ge \vl{\Th_1} + \vl{\Th_2}$
and the distance $(\Th_1,\Th_2) \mapsto 2^{-\vl{\Th_2-\Th_1}}$ is complete.
}
\end{enumerate}

In the saddle-node case we can choose for~$\gF$ a certain $\bA$-subalgebra of $\End_\bA\gA$,
with $\bA=\C[[x]]$ and $\gA=\bA[[y]]=\C[[x,y]]$ (the fact that we deal with
operators which commute with the multiplication by an element of~$\bA$ reflects
the fibred character over~$x$ of the situation),
defined as the set $\gF_{\bA,\nu}$ of operators admitting a valuation \wrt\ the valuation~$\nu$
of~$\gA$, \ie\ the set of all $\Th\in\End_\bA\gA$ for which there exists
$\de\in\Z$ such that
$\nu(\Th f)\ge \nu(f)+\de$ for all $f\in\gA$.
This way, the valuation\footnote{	\label{footval}
For technical reasons, rather than the standard valuation on~$\C[[x,y]]$, we
shall use another monomial valuation, defined by
$\nu(x^m y^n)=4m+n$ (see Section~\ref{secformsumm}).
}~$\nu$ of~$\gA$ induces a complete ring
pseudovaluation~$\val$ on~$\gF_{\bA,\nu}$,
namely $\vl{\Th} = \inf_{f\in\gA\setminus\{0\}} \{ \nu(\Th f) - \nu(f) \}$.

One can thus safely speak of ``formally summable'' families in~$\gF$:
for instance, if for any $\de\in\Z$ the set $\ao \omb\in\cN^\bul \mid
\vl{\cV^\omb \bB_\omb}\le\de \af$
is finite, then the family $(\cV^\omb \bB_\omb)$ is formally summable and
formula~\eqref{eqmouldexponoid} defines an element of~$\gF$.
We repeat our definitions in this general context:
\begin{enumerate}[--]
\item a mould is any map $\cN^\bul \to \bA$
(we usually denote by~$M^\bul$ the mould whose value on the word~$\omb$ is~$M^\omb$),
\item a comould is any map $\cN^\bul \to \gF$
(we usually denote by~$\bB_\bul$ the comould whose value on the word~$\omb$ is~$\bB_\omb$),
\item a mould expansion is the result of the contraction of a mould~$M^\bul$ and a
comould~$\bB_\bul$ such that the family $(M^\omb\bB_\omb)_{\omb\in\cN^\bul}$ is
formally summable in~$\gF$;
we usually use the short-hand notation 
$$
\Th = \sum M^\bul \bB_\bul.
$$
\end{enumerate}

The monoid law in~$\cN^\bul$ is the concatenation, denoted by~$\conc$\,, which
allows us to define \emph{mould multiplication} by the formula
$$
P^\bul = M^\bul \times N^\bul \colon \;
\omb \mapsto P^\omb = \sum_{\omb= \omb^1\!\conc\omb^2} M^{\omb^1} N^{\omb^2}.
$$
The space of moulds is in fact an $\bA$-algebra.
Correspondingly, if a comould~$\bB_\bul$ is \emph{mutiplicative}, in the sense
that
$\bB_{\omb^1\!\conc\omb^2} = \bB_{\omb^2}  \bB_{\omb^1}$ for any
$\omb^1,\omb^2\in\cN^\bul$
(as is obviously the case for a comould defined as in~\eqref{eqdefbBn}), then
$$
\sum \left(M^\bul  \times N^\bul \right) \bB_\bul = 
\left( \sum N^\bul \bB_\bul \right)\left( \sum M^\bul \bB_\bul \right)
$$
as soon as both expressions in the \rhs\ are formally summable.

One can easily check that a mould~$M^\bul$ has a multiplicative inverse if and
only if~$M^\est$ is invertible in~$\bA$.
In this case, if $(M^\omb\bB_\omb)$ is formally summable and $\bB_\bul$ is multiplicative,
then $\sum M^\bul \bB_\omb$ is invertible in~$\gF$.

We do not develop farther the theory here and prefer to return to the formal
normalisation of our saddle-node singularity~$X$.


\section{Solution of the formal conjugacy problem}	\label{secsolconj}


Let us use the $\bA$-algebra $\gF = \gF_{\bA,\nu}$ defined in the previous
section,
to which each operator $B_n = y^{n+1}\pa_y$ obviously belongs
($B_n$ is $\bA$-linear and $\vl{B_n}=n$),
and thus also the $\bB_\omb$'s defined by~\eqref{eqdefbBn}.
Formula~\eqref{eqdefBn} can be written $X-X_0 = \sum J_a^\bul \bB_\bul$, with
\begin{equation}	\label{eqdefJa}
J_a^\omb = \left| \begin{aligned}
a_{n_1} \quad & \text{if $\omb = (n_1)$}\\
0 \ens\; \quad & \text{if $r(\omb)\neq1$}
\end{aligned}  \right.
\end{equation}
and the conjugacy equation~\eqref{eqconjugOp} is equivalent to
\begin{equation}	\label{eqThXXz}
\Th (X-X_0) = [X_0,\Th].
\end{equation}

\begin{lemma}
Let $\cV^\bul$ be a mould such that the family $(\cV^\omb\bB_\omb)$ is formally
summable in~$\gF$. 
Then 
$$
[X_0,\sum \cV^\bul \bB_\bul] = \sum (x^2\pa_x \cV^\bul + \na \cV^\bul) \bB_\bul,
$$
where the mould $\na \cV^\bul$ is defined by $\na \cV^\est=0$ and
$$
\na \cV^\omb = (n_1+\dotsb+n_r) \cV^\omb
$$
for $\omb = (n_1,\dotsc,n_r)$ non-empty.
\end{lemma}

\begin{proof}
The operator $\bB_{n_1,\dots,n_r}$ is homogeneous of degree $n_1+\dotsb+n_r$.
One can check that, if $\Th\in\End_\bA(\bA[[y]])$ is homogeneous of
degree $n\in\Z$, then
$$
[ y\pa_y, \Th ] = n \Th.
$$
Indeed, by $\bA$-linearity and formal continuity, it is sufficient to check that both operators
act the same way on a monomial~$y^k$; 
but
$\Th y^k = \be_{k} y^{k+n}$ with a $\be_{k} \in \bA$,
thus $y\pa_y \Th y^k = (k+n) \be_{k}  y^{k+n} 
= (k+n) \Th y^k$ 
while $\Th y\pa_y y^k = k \Th y^k$.

Since $X_0 = x^2\pa_x + y\pa_y$ and
$x^2\pa_x$ commutes with the $B_n$'s, it follows that
$$
[ X_0, \cV^\omb\bB_\omb ] = \big(x^2\pa_x\cV^\omb+ (n_1+\dotsb+n_r)\cV^\omb) \bB_\omb, \qquad
\omb = (n_1,\dotsc,n_r) \in \cN^\bul.
$$
The conlusion follows by formal continuity.
\end{proof}

Looking for a solution of the form $\Th = \sum \cV^\bul \bB_\bul$ for
equation~\eqref{eqThXXz}, we are thus led to the mould equation
\begin{equation}	\label{eqmldeq}
x^2 \pa_x \cV^\bul + \na\cV^\bul = J_a^\bul \times \cV^\bul.
\end{equation}

\begin{lemma}	\label{lemdefcV}
Equation~\eqref{eqmldeq} has a unique solution~$\cV^\bul$ such that $\cV^\est=1$
and $\cV^\omb\in x\C[[x]]$ for every non-empty $\omb\in\cN^\bul$.
Moreover, 
\begin{equation}	\label{eqvalcV}
\cV^{n_1,\dotsc,n_r} \in x^{\ceil{r/2}}\C[[x]],
\end{equation}
where $\ceil{s}$ denotes, for any $s\in\R$, the least integer not smaller than~$s$.
\end{lemma}

\begin{proof}
Let us perform the change of variable $z=-1/x$ and set
$\pa = \frac{\dd\,}{\dd z}$ and $\wt a_n(z) = a_n(-1/z)$.
Observe that
\begin{equation}	\label{eqassan}
\wt a_n \in z\ii\C[[z\ii]], \qquad \wt a_0 \in z^{-2}\C[[z\ii]],
\end{equation}
as a consequence of~\eqref{eqassA}.

The equation for
$\wt\cV^\omb(z) = \cV^\omb(-1/z)$, with $\omb = (n_1,\dotsc,n_r)$, $r\ge1$, is
\begin{equation}	\label{eqdefwtcV}
(\pa+ n_1 + \dotsb + n_r) \wt\cV^{n_1,\dotsc,n_r}
= \wt a_{n_1} \wt\cV^{n_2,\dotsc,n_r}.
\end{equation}
On the one hand, $\pa+ \mu$ is an invertible operator of $\C[[z\ii]]$
for any $\mu\in\C^*$ and the inverse operator
\begin{equation}	\label{eqdefinvpamu}
(\pa+ \mu)\ii = \sum_{r\ge0} \mu^{-r-1}(-\pa)^r
\end{equation} 
leaves $z\ii\C[[z\ii]]$ invariant;
on the other hand, when $\mu=0$, $\pa$ induces an isomorphism 
$z^{-2}\C[[z\ii]] \to z\ii\C[[z\ii]]$.

For $r=1$, equation~\eqref{eqdefwtcV} has a unique solution in $z\ii\C[[z\ii]]$, 
because the \rhs\ is~$\wt a_{n_1}$, element of $z\ii\C[[z\ii]]$, and even of
$z^{-2}\C[[z\ii]]$ when $n_1=0$.
By induction, for $r\ge2$, we get a \rhs\ in $z^{-2}\C[[z\ii]]$ and a unique solution
$\wt\cV^\omb$ in $z\ii\C[[z\ii]]$ for $\omb=(n_1,\dotsc,n_r)\in\cN^r$.
Moreover, with the notation $`\omb = (n_2,\dotsc,n_r)$, we have
$$
v(\wt\cV^\omb) \ge \al^\omb + v(\wt\cV^{`\omb}), \qquad 
\text{with}\;
\al^\omb = \left| \begin{aligned} 
0 \quad &\text{if $n_1+\dotsb+n_r=0$ and $n_1\neq0$,} \\
1 \quad &\text{if $n_1+\dotsb+n_r\neq0$ or $n_1=0$,} 
\end{aligned}  \right.
$$
where $v$ denotes the standard valuation of $\C[[z\ii]]$.
Thus $v(\wt\cV^\omb) \ge \card \cR^\omb$, with 
$\cR^\omb = \ao i\in[1,r] \mid n_i+\dotsb+n_r\neq0 \;\text{or}\; n_i=0
\af$
for $r\ge1$.

Let us check that $\card\cR^\omb \ge \ceil{r/2}$.
This stems from the fact that if $i\not\in\cR^\omb$, $i\ge2$, then $i-1\in\cR^\omb$
(indeed, in that case $n_{i-1}+\dotsb+n_r = n_{i-1}$),
and that $\cR^\omb$ has at least one element, namely~$r$.
The inequality is thus true for $r=1$ or~$2$; by induction, if $r\ge3$, then
$\cR^\omb \cap [3,r] = \cR^{``\omb}$ with $``\omb = (n_3,\dotsc,n_r)$ and
either $2\in \cR^\omb$, or $2\not\in\cR^\omb$ and $1\in\cR^\omb$, 
thus $\card\cR^\omb \ge 1 + \card\cR^{``\omb}$.
\end{proof}

The formal summability of the family $(\cV^\omb \bB_\omb)$ follows
from~\eqref{eqvalcV} if we use the modified valuation
of footnote~\ref{footval}\label{secformsumm}.
Indeed, one gets 
$\vl{\cV^{n_1,\dotsc,n_r} \bB_{n_1,\dotsc,n_r}}
\ge n_1+\dotsb+n_r + 2r$.
The $n_i$'s may be negative but they are always $\ge-1$, thus
$n_1+\dotsb+n_r + r\ge0$.
Therefore, for any $\de>0$, the condition $n_1+\dotsb+n_r + 2r \le \de$
implies $r\le \de$ and 
$\sum (n_i+1) = n_1+\dotsb+n_r + r \le \de$.
Since this condition is fulfilled only a finite number of times, the summability
follows.

\emph{Setting $\Th = \sum \cV^\bul \bB_\bul$,
we thus get a solution of the conjugacy equation~\eqref{eqconjugOp} and $\Th$ is a
continuous operator of~$\gA$.}

But is~$\Th$ an algebra automorphism?
If one takes for granted the existence of a unique~$\th$ of the
form~\eqref{eqdefthph} which solves~\eqref{eqconjugeq} (and this is not hard to
check), then one can easily identify the operator~$\Th$ that we just defined
with the substitution operator corresponding to~$\th$ and $\Th$ is thus an algebra
automorphism. 
But it is possible to prove directly this fact by checking a certain
symmetry property of the mould~$\cV^\bul$, called \emph{symmetrality}.


\section{Cosymmetrality and symmetrality}


Let us return for a while to the general context of Section~\ref{secgnal}, with
an alphabet~$\cN$ and a commutative $\C$-algebra~$\bA$, focusing on the case
where~$\bB_\bul$ is the multiplicative comould generated by a family of
$\bA$-linear derivations $(B_n)_{n\in\cN}$ of a commutative algebra~$\gA$.

We thus assume that $\gA$ is a commutative $\bA$-algebra, on which a complete ring
pseudovaluation~$\nu$ is given, that $\gF = \gF_{\bA,\nu}$ and that, 
for each $n\in\cN$, we are given $B_n\in\gF$ satisfying the Leibniz rule
$$
B_n(fg) = (B_n f)g + f (B_n g), \qquad f,g\in\gA.
$$
This property can be rewritten
$$
\sig(B_n) = B_n\otimes\ID + \ID\otimes B_n,
$$
with the notation $\sig \colon \End_\bA \gA \to \Bili(\gA\times\gA,\gA)$ for the
composition with the multiplication of~$\gA$ 
and $\Th_1\otimes\Th_2(f,g) := (\Th_1 f)(\Th_2 g)$ for any $\Th_1,\Th_2\in\End_\bA\gA$.

We now consider the comould defined by
$$
\bB_\est = \ID, \qquad
\bB_\omb = B_{n_r}  \dotsm B_{n_1} \quad
\text{for non-empty $\omb = (n_1,\dotsc,n_r) \in \cN^\bul$.}
$$
One can easily check, by iteration of the Leibniz rule, that
\begin{gather*}
\sig\big(\bB_{(n_1,n_2)}\big) = 
\bB_{(n_1,n_2)}\otimes \bB_\est
+ \bB_{(n_1)} \otimes \bB_{(n_2)}
+ \bB_{(n_2)} \otimes \bB_{(n_1)}
+ \bB_\est\otimes \bB_{(n_1,n_2)}, \\[1.5ex]
\begin{split}
\sig\big(\bB_{(n_1,n_2,n_3)}\big) = 
&\bB_{(n_1,n_2,n_3)}\otimes \bB_\est\\
&\quad 
+ \bB_{(n_1,n_2)} \otimes \bB_{(n_3)}
+ \bB_{(n_1,n_3)} \otimes \bB_{(n_2)}
+ \bB_{(n_2,n_3)} \otimes \bB_{(n_1)}\\
&\quad
+ \bB_{(n_3)} \otimes \bB_{(n_1,n_2)} 
+ \bB_{(n_2)} \otimes \bB_{(n_1,n_3)} 
+ \bB_{(n_1)} \otimes \bB_{(n_2,n_3)} \\
& \hspace{19em} + \bB_\est\otimes \bB_{(n_1,n_2,n_3)},
\end{split}
\end{gather*}
the general formula being
\begin{equation}	\label{eqmotivsh}
\sig(\bB_\omb) = \sum_{\omb^1,\omb^2\in\cN^\bul} \sh{\omb^1}{\omb^2}{\omb}
\bB_{\omb^1} \otimes \bB_{\omb^2},
\end{equation}
where $\tsh{\omb^1}{\omb^2}{\omb}$ denotes the number of obtaining~$\omb$ by
\emph{shuffling} of $\omb^1$ and~$\omb^2$:
it is the number of permutations~$\sig$ such that one can write
$\omb^1 = (\om_1,\dotsc,\om_\ell)$,
$\omb^2 = (\om_{\ell+1},\dotsc,\om_r)$ and
$\omb = (\om_{\sig(1)},\dotsc,\om_{\sig(r)})$ 
with the property
$\sig(1)<\dotsm<\sig(\ell)$ and $\sig(\ell+1)<\dotsm<\sig(r)$
(thus it is non-zero only if $\omb$ can be obtained by interdigitating the letters
of~$\omb^1$ and those of~$\omb^2$ while preserving their internal order
in~$\omb^1$ or~$\omb^2$).

Any comould satisfying~\eqref{eqmotivsh} is said to be \emph{cosymmetral}.

Dually, one says that a mould~$M^\bul$ is \emph{symmetral} if and only if
$M^\est = 1$ and, for any two non-empty words $\omb^1, \omb^2$,
\begin{equation}	\label{eqdefsymal}
\sum_{\omb\in\cN^\bul}  \sh{\omb^1}{\omb^2}{\omb} M^\omb = M^{\omb^1} M^{\omb^2}.
\end{equation}
One says that a mould~$M^\bul$ is \emph{alternal} if and only if
$M^\est = 0$ and, for any two non-empty words $\omb^1, \omb^2$, the above sum vanishes.

Suppose that $\bB_\bul$ is cosymmetral and that the family $(M^\omb\bB_\omb)$ is
formally summable, with $\Th = \sum M^\bul\bB_\bul$. 
It is easy to check that
\begin{align*}
M^\bul \ens\text{symmetral} &\quad\Rightarrow\quad
\sig(\Th) = \Th \otimes \Th, \\
M^\bul \ens\text{alternal} &\quad\Rightarrow\quad
\sig(\Th) = \Th \otimes \ID + \ID \otimes \Th.
\end{align*}
In other words, contraction with a symmetral mould yields an automorphism, while
contraction with an alternal mould yields a derivation.

Symmetral and alternal moulds satisfy many stability properties (see \cite{irma}
\S5). 
Here we just mention that the multiplicative inverse~$\wt M^\bul$ of a symmetral
mould~$M^\bul$ is the symmetral mould defined by
\begin{equation}	\label{eqdefinvsym}
\wt M^{n_1,\dotsc,n_r} = (-1)^r M^{n_r,\dotsc,n_1}.
\end{equation}

Since the multiplicative comould generated by a family of derivations is
cosymmetral, in the case of the saddle-node it is sufficient to check the
following lemma to prove that the~$\Th$ defined in Section~\ref{secsolconj} is
indeed an automorphism:

\begin{lemma}
The mould~$\cV^\bul$ defined by Lemma~\ref{lemdefcV} is symmetral.
\end{lemma}

\begin{proof}
We must show that
\begin{equation}	\label{eqVsymal}
\cV^\alb \cV^\beb = \sum_{\gab\in\cN^\bul} \sh{\alb}{\beb}{\gab} \cV^\gab,
\qquad \alb,\beb\in\cN^\bul.
\end{equation}
Since $\cV^\est = 1$, this is obviously true for $\alb$ or $\beb = \est$.
We now argue by induction on $r = r(\alb) + r(\beb)$.
We thus suppose that $r\ge1$ and, without loss of generality, both of $\alb$ and
$\beb$ non-empty.
With the notations $d = x^2\frac{\dd\,}{\dd x}$,
$\norm{\alb}=\al_1+\dotsb+\al_{r(\alb)}$
and $\norm{\beb}=\be_1+\dotsb+\be_{r(\beb)}$,
we compute
\begin{multline*}
A := (d+\norm{\alb}+\norm{\beb}) \sum_\gab \tsh{\alb}{\beb}{\gab} \cV^\gab \\
= \sum_{\gab\neq\est} \tsh{\alb}{\beb}{\gab} (d+\norm{\gab}) \cV^\gab 
= \sum_{\gab\neq\est} \tsh{\alb}{\beb}{\gab} a_{\ga_1} \cV^{`\gab},
\end{multline*}
using the notations $\norm{\gab} = \ga_1+\dotsb+\ga_s$
and $`\gab = (\ga_2,\dotsc,\ga_s)$ 
for any non-empty
$\gab = (\ga_1,\dotsc,\ga_s)$
(with the help~\eqref{eqmldeq} for the last identity).
Splitting the last summation according to the value of~$\ga_1$, 
which must be~$\al_1$ or~$\be_1$, we get
$$
A = \sum_\deb \tsh{`\alb}{\beb}{\deb} a_{\al_1}  \cV^\deb
+ \sum_\deb \tsh{\alb}{`\beb}{\deb} a_{\be_1}  \cV^\deb
= a_{\al_1} \cV^{`\alb}  \cdot \cV^\beb 
+ \cV^{\alb}  \cdot a_{\be_1} \cV^{`\beb}
$$
(using the induction hypothesis), hence, using again~\eqref{eqmldeq},
$$
A = (d+\norm{\alb})\cV^\alb \cdot \cV^\beb 
+ \cV^\alb \cdot (d+\norm{\beb})\cV^\beb
= (d+\norm{\alb}+\norm{\beb}) (\cV^\alb \cV^\beb).
$$
We conclude that both sides of~\eqref{eqVsymal} must coincide, because
$d+\norm{\alb}+\norm{\beb}$ is invertible if $\norm{\alb}+\norm{\beb}\neq0$
and both of them belong to $x\C[[x]]$, thus even if $\norm{\alb}+\norm{\beb}=0$
the desired conclusion holds.
\end{proof}




\section{Resurgence of the formal conjugacy}


At this stage, we have found formal series $\cV^\omb\in\C[[x]]$ which determine
a formally continuous algebra automorphism $\Th=\sum\cV^\bul\bB_\bul$
conjugating~$X_0$ and~$X$.
Since $\Th x = x$, we deduce that $\Th$ is the substitution operator associated
with the formal transformation
$\th(x,y) = \big( x, \ph(x,y) \big)$ where $\ph = \Th y$.
This is what was announced in~\eqref{eqdefthph}. 

The components~$\ph_n(x)$ of~$\ph(x,y)$ are easily computed:
on checks by induction that
\begin{equation}	\label{eqdefbeomb}
\bB_\omb y = \be_\omb y^{n_1+\dotsb+n_r+1}, \qquad 
\omb=(n_1,\dotsc,n_r),\, r\ge 1,
\end{equation}
with $\be_\omb = 1$ if $r=1$,
$\be_\omb = (n_1+1)(n_1+n_2+1)\dotsm(n_1+\dotsb+n_{r-1}+1)$ if $r\ge2$;
one has $\be_\omb=0$ whenever $n_1+\dotsb+n_r \le -2$ (since \eqref{eqdefbeomb}
holds a priori in the fraction field $\C(\!(y)\!)$ but $\bB_\omb y$ belongs to~$\C[[y]]$), hence
\begin{equation}	\label{eqseriesgivphn}
\ph(x,y) = \Th y = y + \sum_{n\ge0} \ph_n(x) y^n, \qquad 
\ph_n = \sum_{\substack{r\ge1, \, \omb\in\cN^r  \\ n_1+\dotsb+n_r+1=n}} 
\be_\omb \cV^\omb
\end{equation}
(in the series giving $\ph_n$, there are only finitely many terms for each~$r$,
\eqref{eqvalcV} thus yields its formal convergence in $x\C[[x]]$).

Similarly, if we define $\cVt^\bul$ as the multiplicative inverse of~$\cV^\bul$
with the help of formula~\eqref{eqdefinvsym},
then $\Th\ii = \sum \cVt^\bul \bB_\bul$ is the substitution operator of
a formal transformation $(x,y) \mapsto \big(x,\psi(x,y)\big)$, which is nothing
but~$\th\ii$, and 
\begin{equation}	\label{eqpsiThii}
\psi(x,y) = \Th\ii y = y + \sum_{n\ge0} \psi_n(x) y^n,
\end{equation} 
where each coefficient can be represented as a formally convergent series
$\dst
\psi_n = \sum_{n_1+\dotsb+n_r+1=n}
\be_\omb \cVt^\omb$.

These are remarkably explicit formulas, quite different from what one would have
obtained by solving directly the differential equation~\eqref{eqdiffeqX}
for $\wt\ph_0(z) = \ph_0(-1/z)$ for instance.

An advantage of these formulas is that they allow to prove the resurgent
character \wrt\ the variable $z=-1/x$ of all the formal series which appear in
our problem.
We recall that a formal series
$$
\wt\ph(z) = \sum_{n\ge0} c_n z^{-n-1} \in z\ii\C[[z\ii]]
$$
is said to be resurgent if its formal Borel transform
$$
\wh\ph(\ze) = \sum_{n\ge0} c_n \frac{\ze^n}{n!} \in \C[[\ze]]
$$
has positive radius of convergence and defines a holomorphic function of~$\ze$
which admits an analytic continuation along all the paths starting in its
disc of convergence and lying in $\C\setminus\Z$
(see \cite{Eca81}, or \cite{kokyu}, or \cite{irma} \S8).
This property is stable by multiplication, the Borel transform of the Cauchy product
$\wt\ph \cdot\wt\psi$ being the convolution product $\wh\ph*\wh\psi$ defined by
\begin{equation}	\label{eqdefconvol}
(\wh\ph * \wh\psi)(\ze) = \int_0^\ze \wh\ph(\ze_1) \wh\psi(\ze-\ze_1)
\qquad \text{for $|\ze|$ small enough}
\end{equation}
(with $\wh\ph$ and~$\wh\psi$ denoting the Borel transforms of~$\wt\ph$ and~$\wt\psi$ respectively).

A simple example is provided by the formal series $\wt\cV^\omb(z) =
\cV^\omb(-1/z)$:
indeed, the Borel transforms of the convergent series $\wt a_n(z) = a_n(-1/z)$
are entire functions~$\wh a_n(\ze)$ and
the Borel counterpart of $\pa = \frac{\dd\,}{\dd z}$ is multiplication
by~$-\ze$, hence \eqref{eqdefwtcV} yields
\begin{gather*}
\wh\cV^{n_1}(\ze) = -\frac{1}{\ze-n_1} \wh a_{n_1}(\ze) \\[1ex]
\wh\cV^{n_1,n_2}(\ze) = \frac{1}{\ze-(n_1+n_2)} \big(\wh a_{n_1}*\wh\cV^{n_2} \big) 
\\[1ex]
\vdots \\[1ex]
\wh\cV^{n_1,\dotsc,n_r} = (-1)^r
\frac{1}{\ze-\htb n_1} \Big( \wh a_{n_1}  * 
\Big( \frac{1}{\ze-\htb n_2} \Big( \wh a_{n_2}  * 
\Big(  \dotsb 
\Big( \frac{1}{\ze-\htb n_r} \wh a_{n_r}
\Big) \dotsm \Big)\Big)\Big)\Big)
\end{gather*}
with $\htb n_i = n_i + \dotsb + n_r$.
Since the $\wh a_n$'s are entire functions of~$\ze$, the function
$\wh\cV^{n_1,\dotsc,n_r}$ is holomorphic with $\htb n_1,\dotsc,\htb n_r$ as only
possible singularities.

Moreover, the above formula is sufficiently explicit to make it possible to give
majorant series arguments so as to prove the uniform convergence of the series
of holomorphic functions
$$
\wh\ph_n = \sum_{n_1+\dotsb+n_r+1=n}
\be_\omb \wh\cV^\omb, \qquad
\wh\psi_n = \sum_{n_1+\dotsb+n_r+1=n}
\be_\omb \wh\cVt^\omb
$$
(\S8 of \cite{irma} is devoted to this task); in other words, \emph{the formal series $\ph_n$ and $\psi_n$
are resurgent \wrt\ $z=-1/x$}.


\section{Conclusion}


\parage
The reader is referred to \cite{Eca84} or \cite{irma} \S9--11 for the resurgent
approach to the question of the analytic classification of saddle-node
singularities, which one can develop once the analytic structure of the 
functions~$\wh\ph_n(\ze)$ is clear.
This approach relies on the use of \'Ecalle's \emph{alien calculus}: the singularities in
the $\ze$-plane are controlled through operators~$\De_m$, $m\in\Z^*$, called
alien derivations; 
mould calculus allows one to summarize the singular structure of the
functions $\wh\cV^\omb$ in the simple equation
$$
\De_m \wt\cV^\bul = \wt\cV^\bul \times V^\bul(m),
$$
where $V^\bul(m)$ is an alternal scalar-valued mould (the $V^\omb$'s are complex
numbers), 
with $n_1+\dotsb+n_r\neq m \;\Rightarrow\; V^{n_1,\dotsc,n_r}(m)=0$.

This leads to \'Ecalle's ``Bridge Equation'' for~$\Th$, which gives the alien
derivatives of all the resurgent functions~$\wt\ph_n$ and
in which the resurgent solution of the analytic classification problem is subsumed.

\parage
Another application of mould calculus in Resurgence theory is the use of
the formal series~$\wt\cV^\omb$ to construct ``resurgence monomials'' $\wt\cU^\omb$ which
behave as simply as possible under alien derivation.
This allows one to prove that alien derivations generate an infinite-dimensional
free Lie algebra.

\parage
Other problems of formal normalisation can be handled by an approach similar to
the one explained in Section~\ref{secnormalmdlexp}.
See \cite{irma} \S13 for the simple case of the linearisation of a vector field
with non-resonant spectrum;
much more complicated situations (taking into account resonances) are considered
in \cite{ES} or \cite{EVcorr}.
The strategy always begins by expanding the object to study in a sum of
homogeneous components~$B_n$ and considering the corresponding multiplicative
comould~$\bB_\bul$. 

It must be mentioned that, when this is applied to the analysis of a local
diffeomorphism, rather than a local vector field, the operators~$B_n$ are no
longer derivations, but they still satisfy a kind of modified Leibniz rule:
$$
\sig(B_n) = B_n \otimes \ID +
\sum_{n'+n''=n} B_{n'}\otimes B_{n''} 
+ \ID \otimes B_n.
$$
The resulting comould is not cosymmetr\hspace{-1pt}{\em a}l but cosymmetr\hspace{-1pt}{\em e}l, a property
which involves ``contracting shuffling coefficients'' instead of the shuffling
coefficients $\tsh{\omb^1}{\omb^2}{\omb}$. 
Correspondingly, the relevant symmetry properties for the moulds to be
contracted into~$\bB_\bul$ are ``symmetrelity'' and ``alternelity''.

The theory can thus be extended so as to treat on an equal footing vector fields
and diffeomorphisms.

\vfill
\pagebreak



\frenchspacing
\begin{thebibliography}{7}







\bibitem{Eca81} J. \'Ecalle, 
%
\textit{Les fonctions r\'esurgentes},
%
Publ. Math. d'Orsay [Vol.~1: 81-05, Vol.~2: 81-06, Vol.~3: 85-05]
%
1981, 1985.


\bibitem{Eca84} J. \'Ecalle,
%
\textit{Cinq applications des fonctions r\'esurgentes},
%
Publ. Math. d'Orsay 84-62, 1984.














\bibitem{ES} J.~\'Ecalle and D.~Schlomiuk,
%
The nilpotent part and distinguished form of resonant vector fields or
diffeomorphisms,
%
\emph{Ann. Inst. Fourier (Grenoble)}  43  (1993),  no. 5, 1407--1483.


\bibitem{EVcorr} J.~\'Ecalle and B.~Vallet,
%
Correction and linearization of resonant vector
fields and diffeomorphisms, 
%
\emph{Math. Z.} 229 (1998),  no. 2, 249--318.


\bibitem{MR} J. Martinet and J.-P. Ramis,
%
Probl\`emes de modules pour des \'equations diff\'erentielles non lin\'eaires du
premier ordre,
%
\emph{Publ. Math. Inst. Hautes \'Etudes Sci.} 55 (1982), 63--164.


\bibitem{Moussu} R.~Moussu,
%
Singularit\'es d'\'equations diff\'erentielles holomorphes en dimension deux,
in \emph{Bifurcations and periodic orbits of vector fields (Montreal, PQ,
1992)} (ed. by D. Schlomiuk), NATO Adv. Sci. Inst. Ser.C Math. Phys. Sci. 408, Kluwer
Acad. Publ., Dordrecht 1993, 321--345.


\bibitem{kokyu} D. Sauzin, 
%
Resurgent functions and splitting problems,
%
\emph{RIMS Kokyuroku} 1493 (2005), 48--117.


\bibitem{irma} D. Sauzin, 
%
Mould expansions for the saddle-node and resurgence monomials,
%
{\em to appear in the
%
proceedings of the international conference on Renormalization and Galois
theories (CIRM, Luminy, 13--17 March 2006)} (ed. by A. Connes, F. Fauvet, J.-P. Ramis),
%
IRMA Lectures in Mathematics and Theoretical Physics, 2008.
%
\url{http://hal.archives-ouvertes.fr/hal-00197145/fr}


\end{thebibliography}
\end{document}